\documentclass[10pt,twocolumn,twoside]{IEEEtran} 
\usepackage{xcolor,colortbl}
\definecolor{LightCyan}{rgb}{0.88,1,1}

\usepackage{graphics} 
\usepackage{epsfig} 
\usepackage{amsmath} 
\usepackage{amsthm,amsfonts}
\usepackage{amssymb}  

\usepackage{mathrsfs}

\usepackage{subfigure}
\usepackage{color} 
\usepackage{cite}
\usepackage{mathtools}
\usepackage{mathrsfs}
\usepackage{txfonts}
\usepackage{amssymb}
\usepackage{multirow}
\usepackage{subfigure}

\newtheorem{assumption}{Assumption}[section]

\newcommand{\HEI}[1]{\bf Hybrid-EI}

\newtheorem{theorem}{Theorem}
\newtheorem{corollary}{Corollary}
\newtheorem{remark}{Remark}

\ifCLASSINFOpdf

\else

\fi

\begin{document}
%
\title{\LARGE \textbf{Synchronization of Complex Dynamical Networks via Event-Triggered Pinning Impulses}}
%
%
%

\author{Kexue~Zhang 
\thanks{This work was supported by the Natural Sciences and Engineering Research Council of Canada (NSERC) under the grant {RGPIN-2022-03144}.}
\thanks{K. Zhang is with the Department
of Mathematics and Statistics, Queen's University, Kingston, Ontario K7L 3N6, Canada (e-mail: kexue.zhang@queensu.ca).}
 }

\maketitle
\thispagestyle{empty}
%
\begin{abstract}
This article studies the synchronization problem of complex dynamical networks. The impulsive control method is considered with a novel event-triggered pinning algorithm. Sufficient conditions on the network topology are obtained to ensure network synchronization. It is shown that synchronization can be realized with a careful selection of the pinning nodes. Furthermore, an adaptive coupling strength is incorporated into the network to allow network synchronization with an arbitrary selection of the pinning nodes. An example of a network with node dynamics described by the Chen system is studied to demonstrate the theoretical results.

\end{abstract}
%

\section{Introduction}\label{Sec1}

Complex dynamical networks (CDNs) consist of a large collection of nodes, which are normally modeled by dynamical systems, and these nodes are connected according to certain topological structures. Many interconnected systems in nature can be described by CDNs (see, e.g., \cite{EE:2012,CWW:2007}), such as, multi-agent systems, electrical power grids, the World Wide Web, biological and artificial neural networks, etc. As a typical collective behavior, synchronization of CDNs has been investigated extensively due to its wide applications in various scientific fields ranging from biology, engineering to physics and sociology (see, e.g., \cite{FD-FB:2014,YT-FQ-HG-JK:2014}). 

Numerous synchronization problems have been investigated with a wide variety of control methods. Among them, impulsive control paradigm has been proved to be powerful and efficient in network synchronization \cite{CWW:2007,YT-FQ-HG-JK:2014}. As a particular type of feedback control, impulsive control uses impulses, which are state abrupt changes or jumps at a sequence of discrete times, to achieve the network synchronization \cite{TY:2001,BM-EYR:2003}. The sequence of these impulse times is normally determined in two fashions. If the impulse times are pre-scheduled, the corresponding control method is called time-triggered impulsive control. Sufficient conditions on the upper bound of the times between two consecutive impulses are naturally required so that the impulses occur frequently enough to realize the network synchronization (see \cite{XL-KZ:2019} and references therein). Another way to trigger an impulse is via a certain event which is the violation of some well-designed triggering condition by the network synchronization error. Thus, the impulse times are implicitly defined by the event occurrence. Compared with time-triggered impulsive control, potentially less impulses can be activated by the event-triggered impulsive control as the impulses are triggered only when needed. To ensure the feasibility of the event-triggered impulses, it is vital important to exclude Zeno behavior, a phenomenon of infinitely many impulses over a finite time period, from CDNs with impulses. Recently, many event-triggering algorithms have been successfully integrated with the impulsive control strategy for stabilization of nonlinear systems (see, e.g., \cite{KZ-EB:2022,XL-XY-JC:2020,BJ-JL-XL-JQ:2021}) and synchronization of CDNs (see, e.g., \cite{XL-JC-XL-MA-UAA:2020,KU-FAR-XL-RR:2021,SS-KU-DG-RR:2022}).

To achieve the network synchronization, high dimensionality as a typical characteristic of CDNs makes it costly and sometimes impractical to control all the network nodes. Therefore, the paradigm to control a portion of nodes, known as pinning control, has received increasing interests in the past few decades \cite{ROG-MCC-GHS:1997,WY-GC-JL:2009,YO-MJ-XY:2016,PD-FG-FLI:2018}. Pinning impulsive control uses impulses on a selection of nodes to realize the network synchronization, and the collection of nodes to be pinned normally should be deliberately selected so that the impulses can tame the dynamics of every node with the help of the network topological structure. A handful of pinning schemes have been discussed with the impulsive control approach for the stabilization and synchronization problems \cite{JL-JK-JC-NM-CH:2011,JZ-QW-LX:2011}. Most of the existing results focused on time-triggered pinning impulses, and pinning impulsive control approach with event-triggered impulses has started to gain attention very recently (see, e.g., \cite{YS-XL:2022,WL-SP-ZF-TC-ZG:2022}).

Motivated by the above discussion, we study the synchronization problem for a class of CDNs with pinning impulsive control approach. A novel event-triggering algorithm is designed to determine the impulse times, and sufficient conditions on the network topology, impulsive control gains, and parameters in the event-triggering conditions are derived to guarantee the network synchronization. The sufficient condition on the network topology provides a guideline for the selection of pinning nodes. To allow arbitrary pinning schemes for the event-triggered impulsive controller, an adaptive tuning method on the network coupling strength is then introduced. With the adaptive coupling strength, the synchronization of CDNs can be realized via the proposed control method with an arbitrary selection of the pinning nodes. In comparison with the control methods in~\cite{YS-XL:2022,WL-SP-ZF-TC-ZG:2022}, the proposed event-triggering condition for each pinning node depends only on the node's synchronization error, while the triggering conditions in~\cite{YS-XL:2022,WL-SP-ZF-TC-ZG:2022} require the knowledge of the entire network's synchronization errors so that the norms of these error states can be sorted into the descending order. Compared with event-triggered pinning control (see, e.g., \cite{AA-FA-DL-GS-DVD-MDB-KHJ:2015,TW-KS-SS:2020}), event-triggered pinning impulses are instantaneous control inputs at a sequence of discrete moments, and the control system runs open-loop between two consecutive impulse times.

The rest of this paper is organized as follows. Section~\ref{Sec2} formulates the synchronization problem and introduces the event-triggering algorithm. Synchronization criteria are constructed in Section~\ref{Sec3}, and a simulation example is provided in Section~\ref{Sec4} to verify the  main results. Finally, some conclusions and possible future research are drawn in Section~\ref{Sec5}.

\emph{Notation.} Let $\mathbb{R}$ denote the set of real numbers, and denote $\mathbb{R}^n$ and $\mathbb{R}^{n \times n}$, respectively, the $n$-dimensional and $n\times n$-dimensional Euclidean space equipped with the Euclidean norm $\|\cdot\|$. Let $I_n\in\mathbb{R}^{n\times n}$ denote the $n\times n$ identity matrix. For matrices $A\in\mathbb{R}^{m\times n}$ and $B\in\mathbb{R}^{p\times q}$, the transpose of $A$ is $A^\top$, and let $A\otimes B\in\mathbb{R}^{mp\times nq}$ denote their Kronecker product. Let $\delta(t)$ represent the Dirac delta function which is infinity at $t=0$ and zero elsewhere, and it also satisfies the identity $\int^{\infty}_{-\infty} \delta(t) \mathrm{d}t=1$.

\section{Problem Formulation}\label{Sec2}

Consider a complex dynamical network with $N$ identical nodes

\begin{equation}\label{network}
\dot{x}_i(t)=f(t,x_i(t)) + c\sum^{N}_{j=1} a_{ij} \Gamma \left( x_j(t)-x_i(t) \right),~~i=1,2,...,N,
\end{equation}
where $x_i=(x_{i1},x_{i2},...,x_{in})^\top\in\mathbb{R}^n$ represents the state of the $i$th node, function $f:\mathbb{R}^+\times \mathbb{R}^n\mapsto \mathbb{R}^n$ is continuously differentiable, constant $c>0$ is the network coupling strength, and $\Gamma\in\mathbb{R}^{n\times n}$ denotes the inner coupling matrix which is assumed to be positive definite. Matrix $A=( a_{ij} )_{N \times N}$ characterizes the network topological structure, and each entry $a_{ij}$ is defined as follows:
\begin{itemize}
\item If $i\not=j$, then $a_{ij}=a_{ji}$, and they are positive if there exists a connection between nodes $i$ and $j$; otherwise, they are both zeros.
\item The diagonal entries of $A$ satisfy $a_{ii}=-\sum^N_{j=1,j\not=i}a_{ij}$, and $|a_{ii}|$ is called the degree of the $i$th node.

\end{itemize}
The definition of diagonal entries implies $\sum^N_{j=1} a_{ij}=0$, then network~\eqref{network} is equivalent to the following description:
\begin{equation}\label{network1}
\dot{x}_i(t)=f(t,x_i(t)) + c\sum^{N}_{j=1} a_{ij} \Gamma  x_j(t),~~i=1,2,...,N.
\end{equation}

The dynamics of an isolated node is described by the following system:
\begin{equation}\label{node}
\dot{z}(t)=f(t,z(t))
\end{equation}
where $z\in\mathbb{R}^n$ denotes the node state, and the continuously differentiable function $f$ satisfies the following assumption:
\begin{assumption}\label{Assumption}
Suppose that there exists a constant matrix $K$ such that
\begin{equation}\label{assumption}
(x-y)^{{\top}}\left(f(t,x)-f(t,y)\right)\leq (x-y)^\top K\Gamma(x-y)
\end{equation}
for any $x,y\in\mathbb{R}^n$.
\end{assumption} 
See \cite{WY-GC-JL:2009} for examples of various systems and CDNs that satisfy this assumption.

Synchronization problem of network~\eqref{network} will be investigated in the following sense:
\begin{equation}\label{synchronization}
\lim_{t\rightarrow \infty} \|x_i(t)-z(t)\|= 0,~~i=1,2,...,N.
\end{equation}
We say network~\eqref{network} is synchronized with the isolated node~\eqref{node}, or network~\eqref{network} achieves synchronization for short.

In this study, impulsive control method is applied to realize the network synchronization. Instead of controlling all the network nodes, only a portion of nodes are controlled by the impulses. We assume $l$ nodes are controlled by the impulses ($1\leq l\leq N$), and the rest $N-l$ nodes are free of impulses. Rearrange the node indexes so that the first $l$ nodes are to be controlled, then the pinning impulsive control network can be described as follows:

\begin{align}\label{network-pinning}
\dot{x}_i(t)&=f(t,x_i(t)) + c\sum^{N}_{j=1} a_{ij} \Gamma x_j(t) + u_i(t),~~i=1,2,...,l,\cr
\dot{x}_i(t)&=f(t,x_i(t)) + c\sum^{N}_{j=1} a_{ij} \Gamma x_j(t),~~i=l+1,l+2,...,N,
\end{align}
where
\begin{equation}\label{IC}
u_i(t)=-\sum_{k\in\mathbb{N}}d_i (x_i(t)-z(t))\delta(t-t^{(i)}_k),~~i=1,2,...,l
\end{equation}
are the impulsive controllers, $d_i\in\mathbb{R}$ are impulsive control gains, and $\{t^{(i)}_k\}_{k\in\mathbb{N}}$ are the sequences of impulse times to be determined. It should be noted that different nodes could have different sequences of impulse times.

Define synchronization error $e_i(t)=x_i(t)-z(t)$, then the error dynamics can be described as follows:
\begin{align}\label{network-error}
\dot{e}_i(t)&=f(t,x_i(t))-f(t,z(t)) + c\sum^{N}_{j=1} a_{ij} \Gamma e_j(t),~t\not=t^{(i)}_k,\cr
\Delta e_i(t)&= -d_i e_i(t),~t=t^{(i)}_k,~~i=1,2,...,l~\textrm{and}~k\in\mathbb{N},\cr
\dot{e}_i(t)&=f(t,x_i(t))-f(t,z(t)) + c\sum^{N}_{j=1} a_{ij} \Gamma e_j(t),~i=l+1,...,N.
\end{align}
where $\Delta e_i(t)=e_i(t^+)-e_i(t^-)$ represents the error state jump induced by the impulsive control input $u_i$ at the impulse time; $e_i(t^+)$ and $e_i(t^-)$ are the right and left limits of $e_i$ at time $t$, respectively. Throughout this research, we assume the network states are left-continuous at each impulse time, and hence, $e_i(t^-)=e_i(t)$ for all $t\geq t_0$. See \cite[Section~4.1]{XL-KZ:2019} for details of deriving the closed-loop system~\eqref{network-error} from the impulsive control system~\eqref{network-pinning} with~\eqref{IC}.

Next, define $V(t)=\sum^N_{i=1}V_i(t)$ with $V_i(t)=e^\top_i(t)e_i(t)$. To derive the sequence $\{t^{(i)}_k\}_{k\in\mathbb{N}}$ for node $i$ ($i=1,2,...,l$), we enforce $V_i$ to satisfy $V_i(t) < \alpha_i \exp\left({-\beta_i(t-t_0)}\right)$ for all $t\geq t_0$, where $\alpha_i$ and $\beta_i$ are positive constants, and $t_0$ is the initial time. Then, the impulsive control input $u_i$ is activated by the event
\begin{equation}\label{event0}
V_i(t) \geq \alpha_i \exp({-\beta_i(t-t_0)}),
\end{equation}
and the impulse/event times are the moments when the impulse/event occurs:
\begin{equation}\label{event}
t^{(i)}_{k+1}=\inf\left\{t>t^{(i)}_k:~V_i(t)\geq \alpha_i \exp({-\beta_i(t-t_0)}) \right\},~~i=1,2,...,l.
\end{equation}
Assume $V_i(t_0)<\alpha_i$ for $i=1,2,...,l$ so that the first event time $t^{(i)}_1$ is larger than the initial time $t_0$. The pinning impulsive controller for the $i$th node works as follows. Choose parameter $\alpha_i$ large enough so that $V_i(t_0)<\alpha_i$, then the event time $t^{(i)}_1$ is the first moment when $V_i(t)$ reaches the threshold $\alpha_i \exp({-\beta_i(t-t_0)})$. The impulse at this time should be properly designed to bring the value of $V_i$ down below the threshold. The second event time $t^{(i)}_2$ is determined similarly once $V_i(t)$ attains the threshold again. This process can be continued as long as $V_i$ goes over the threshold. If $V_i$ stays below the threshold after some event time, then there are finite events for node $i$. It can be seen that synchronization of the $i$th node with the isolate node can be achieved under the proposed algorithm provided $d_i\in(0,1)$. However, to ensure the entire network synchronization, sufficient conditions on the network topology should be established so that the rest $N-l$ nodes will also achieve synchronization with the isolated node.

\section{Synchronization Criteria}\label{Sec3}
In this section, some sufficient conditions are derived to guarantee the synchronization of network~\eqref{network-pinning}.
\begin{theorem}\label{theorem}
Suppose $\alpha_i>V_i(t_0)$ and $0<d_i<1$ for $i=1,2,...,l$, and denote $\gamma=\|K\|$. Synchronization of the network~\eqref{network-pinning} can be achieved if
\begin{equation}\label{condition2}
\Lambda<0
\end{equation}
where $\Lambda=\gamma I_{N-l}+c \bar{A}$ and $\bar{A}=(\bar{a}_{ij})_{(N-l)\times(N-l)}$ with $\bar{a}_{ij}=a_{l+i,l+j}$. Moreover, network~\eqref{network-pinning} does not exhibit Zeno behavior, that is, $\lim_{k\rightarrow \infty} t^{(i)}_k=\infty$ for each $i=1,2,...,l$.
\end{theorem}

\begin{proof} The proposed event-triggered pinning algorithm ensures that
\begin{equation}\label{Vipin}
V_i(t)\leq \alpha_i \exp\left({-\beta_i(t-t_0)}\right),~~i=1,2,...,l,
\end{equation}
that is, \eqref{synchronization} holds for $i=1,2,...,l$. Then, it is sufficient to show~\eqref{synchronization} holds for $i=l+1,l+2,...,N$. To do this, let $W(t)=\sum^N_{i=l+1} {V}_i(t)$ and $\bar{e}=(e^\top_{l+1},e^\top_{l+2},...,e^\top_{N})^\top$, then the derivative of $W$ along the trajectories of network~\eqref{network-pinning} yields
\begin{align}\label{derivative1}
\dot{W}(t)=& \sum^N_{i=l+1} 2e^\top_i(t)\dot{e}_i(t) \cr
= &\sum^N_{i=l+1} 2e^\top_i \left( f(t,x_i)-f(t,s) +c \sum^N_{j=l+1} a_{ij}\Gamma e_j + c \sum^l_{j=1} a_{ij}\Gamma e_j \right)\cr
\leq &2\bar{e}^\top \left( I_{N-l}\otimes(K\Gamma) + c \bar{A}\otimes\Gamma \right)\bar{e}+2\sum^N_{i=l+1} \sum^l_{j=1} \left(  c a_{ij} e^\top_i \Gamma e_j\right)\cr
\leq & 2\bar{e}^\top \left( I_{N-l}\otimes(K\Gamma) + c \bar{A}\otimes\Gamma + \varepsilon I_{n(N-l)} \right)\bar{e}  + \frac{\theta}{l} \sum^l_{j=1}e^\top_je_j \cr
\leq & 2\bar{e}^\top \left( \Lambda\otimes\Gamma + \varepsilon I_{n(N-l)} \right)\bar{e} + \theta \max_{1\leq j\leq l}\{V_j\} 
\end{align}
where constants $\varepsilon>0$ and 
$$\theta=\frac{1 }{2\varepsilon} (N-l)c^2 \gamma^2 l^2 \max_{{\substack{l+1\leq i\leq N \\ 1\leq j\leq l}}}\{a^2_{ij}\}.$$ 
Since $\Lambda<0$ and $\Gamma>0$, we choose a small enough $\varepsilon>0$ such that $\Lambda\otimes\Gamma + \varepsilon I_{n(N-l)}<0$, and then select a positive constant $\mu$ close to zero so that
\[
2\left(\Lambda\otimes\Gamma + \varepsilon I_{n(N-l)} \right)<-\mu.
\]
From~\eqref{derivative1} and~\eqref{Vipin}, we can conclude that
\begin{equation}\label{detivativeW}
\dot{W}(t)\leq -\mu W(t)+ \hat{\alpha} \exp({-\check{\beta}(t-t_0)}) \textrm{~for~} t\geq t_0,
\end{equation}
where $\hat{\alpha}=\theta\max_{1\leq j\leq l}\{\alpha_j\}$ and $\check{\beta}=\min_{1\leq j\leq l}\{\beta_j\}$. Without loss of generality, we assume $\mu\leq \check{\beta}$. Note that this can always be achieved by selecting small enough $\mu>0$ with the given parameters $\{\beta_j\}^l_{j=1}$.
Multiply both sides of~\eqref{detivativeW} by $\exp({\mu(t-t_0)})$, and the product rule of differentiation concludes
\begin{equation}
\frac{\mathrm{d}}{\mathrm{d}t}\left( \exp\left({\mu(t-t_0)}\right) W(t) \right)\leq \hat{\alpha} \exp({(\mu-\check{\beta})(t-t_0)}),
\end{equation}
then integrating both sides from $t_0$ to $t$ yields
\begin{align}
\exp({\mu(t-t_0)}) W(t) -W(t_0) \leq \frac{\hat{\alpha}}{\check{\beta}-\mu} \left( 1-\exp({-(\check{\beta}-\mu)(t-t_0)}) \right)
\end{align}
that is,
\begin{align}
W(t)\leq \exp({-\mu(t-t_0)}) \left( W(t_0) +\frac{\hat{\alpha}}{\check{\beta}-\mu} \right)
\end{align}
which implies
\begin{equation}\label{synchronization2}
\lim_{t\rightarrow\infty} \|e_i(t)\|=0 \textrm{~for~} i=l+1,l+2,...,N.
\end{equation}
Therefore, we conclude from~\eqref{Vipin} and~\eqref{synchronization2} that synchronization of network~\eqref{network-pinning} can be achieved.

Next, we will show the event-triggered impulsive control network~\eqref{network-pinning} is free of Zeno behavior.

For $t\not= t^{(i)}_k$ with $i=1,2,...,l$ and $k\in\mathbb{N}$, we have
\begin{align}
 \dot{V}_i(t) =&  2e^\top_i\left(  f(t,x_i)-f(t,s)+c\sum^N_{j=1}a_{ij}\Gamma e_j  \right)\cr
          \leq & 2e^\top_i K\Gamma e_i +2 c \sum^N_{j=1}e^\top_i a_{ij}\Gamma e_j \cr
          \leq & 2\|\Gamma\| \left( \gamma V_i(t) -c a_{ii} \|e_i\| \sum^N_{j=1}\|e_j\| \right) \cr
             = & 2\|\Gamma\| \left( \gamma V_i(t) -c a_{ii} \sqrt{V_i(t)} \left( \sum^l_{j=1} \sqrt{V_j(t)}  + \sum^N_{j=l+1} \sqrt{V_j(t)}  \right) \right) \cr
          \leq & 2\|\Gamma\| \left( \gamma \alpha_i -c a_{ii} \sqrt{\alpha_i} \left( \sum^l_{j=1} \sqrt{\alpha_j}  + (N-l)\sqrt{M}  \right) \right) 
\end{align}
where $M= W(t_0) +{\hat{\alpha}}{(\check{\beta}-\mu)^{-1}}$. Define the right-hand side of the above inequality as $\sigma_i$  ($i=1,2,...,l$) which are positive constants since $a_{ii}<0$. Then, for $t\in(t^{(i)}_k,t^{(i)}_{k+1})$, we have
\begin{equation}\label{Vit}
V_i(t)\leq  V_i({t^{(i)+}_k}) + \sigma_i(t-t^{(i)}_k) =  (1-d_i)^2V_i({t^{(i)}_k}) + \sigma_i(t-t^{(i)}_k)
\end{equation}
According to the definition of the event time $t^{(i)}_k$ in~\eqref{event}, we have $V_i({t^{(i)}_k})=\alpha_i \exp({-\beta_i(t^{(i)}_k-t_0)})$. By the fact that $(1-d_i)^2V_i({t^{(i)}_k})< V_i({t^{(i)}_k})$, the following equation of $t$ has a unique solution
\begin{equation}\label{equation}
(1-d_i)^2V_i({t^{(i)}_k}) + \sigma_i(t-t^{(i)}_k) = \alpha_i \exp({-\beta_i(t-t_0)})
\end{equation}
for $t>t^{(i)}_k$, denoted as $t^{(i)}_k+T^{(i)}_k$ with some $T^{(i)}_k>0$. From~\eqref{Vit} and~\eqref{Vipin}, we have $t^{(i)}_{k+1}\geq t^{(i)}_k+T^{(i)}_k$, that is, 
\begin{equation}\label{lowerbdd}
t^{(i)}_{k+1} - t^{(i)}_k\geq T^{(i)}_k>0
\end{equation}
with $i=1,2,...,l$. 

It should be noted that the lower bound $T^{(i)}_k$ of the inter-event time $t^{(i)}_{k+1} - t^{(i)}_k$ depends on the event time $t^{(i)}_k$. Thus, a uniform lower bound of the inter-event times $\{t^{(i)}_{k+1} - t^{(i)}_k\}_{k\in\mathbb{N}}$ can not be guaranteed by~\eqref{lowerbdd}. To show that network~\eqref{network-pinning} does not exhibit Zeno behavior, we use the following contradiction argument. Suppose that there exist some $i$ and a finite number $\bar{t}^{(i)}>0$ such that $\lim_{k\rightarrow \infty} t^{(i)}_k=\bar{t}^{(i)}$, that is, there are infinitely many state jumps for the $i$th node over the finite time interval $[t_0, \bar{t}^{(i)}]$. We then conclude from~\eqref{lowerbdd} that 
\[
\lim_{k\rightarrow \infty} t^{(i)}_{k+1} - t^{(i)}_k =\lim_{k\rightarrow \infty} T^{(i)}_k =0.
\]
From~\eqref{equation} and the definition of $T^{(i)}_k$, it can be derived that
\begin{equation}\label{totakelimit}
(1-d_i)^2 \alpha_i+ \sigma_i T^{(i)}_k \exp(\beta_i (t^{(i)}_k-t_0 )) = \alpha_i \exp(-\beta_i T^{(i)}_k).
\end{equation}
Since the event times $\{t^{(i)}_k\}_{k\in\mathbb{N}}$ are upper bounded by $\bar{t}^{(i)}$ and $0<d_i<1$, letting $k$ go to infinity leads to the contradiction that $(1-d_i)^2 \alpha_i= \alpha_i$. Therefore, network~\eqref{network-pinning} with the proposed event-triggering algorithm does not exhibit Zeno behavior.
\end{proof}

\begin{remark}
It's worth mentioning that if condition~\eqref{condition2} is satisfied for network~\eqref{network}, then the pinning feedback control method proposed in~\cite{WY-GC-JL:2009} can be applied to this network. Nevertheless, the proposed pinning impulsive control method have advantages over the control method in~\cite{WY-GC-JL:2009} in the following sense: 
\begin{itemize}
\item The $i$th node ($i=1,2,...,l$) of network~\eqref{network-pinning} runs open-loop between two consecutive impulses, while these nodes of the pinning control networks in~\cite{WY-GC-JL:2009} evolve closed-loop for all $t\geq t_0$.

\item The impulsive control gain $d_i$ characterizes the jump of error state $e_i$ at each impulse time, and any $d_i\in (0,1)$ will synchronize the network if $\Lambda<0$. However, the feedback control gains in~\cite{WY-GC-JL:2009} should be carefully determined to guarantee the network synchronization.

\item The impulse time sequences $\{t^{(i)}_k\}_{k\in\mathbb{N}}$ for different nodes are generally different, that is, not all these $l$ nodes are controlled at each impulse time. The pinning control method in~\cite{WY-GC-JL:2009} requires all these $l$ nodes to be controlled simultaneously for all $t\geq t_0$.
\end{itemize}
\end{remark}

Condition~\eqref{condition2} not only ensures the network synchronization but also provides the selection algorithm to derive these $l$ nodes to be pinned. To ensure $\Lambda<0$, it is necessary to require its diagonal entries negative, that is, $\gamma +c \bar{a}_{ii}<0$ for $i=l+1,l+2,...,N$. This implies that the degree of the $i$th node must be large enough to satisfy $-\bar{a}_{ii}>\gamma/c$. {Hence, the nodes with degrees not bigger than $\gamma/c$ should be controlled by the event-triggered impulses. To select the correct amount of nodes to be pinned, we can start with these nodes and define $\bar{A}$ accordingly. If $\gamma I+ c\bar{A}< 0$, then synchronization of network~\eqref{network-pinning} can be achieved. Otherwise, add node(s) and check if $\gamma I+ c\bar{A}< 0$ holds with $\bar{A}$ updated.} Repeat this process until condition~\eqref{condition2} is satisfied. See the example in Section~\ref{Sec4} for a detailed demonstration of the selection procedure. The worst scenario of this selection is that all the nodes need to be controlled by the event-triggered impulses, then the following corollary guarantees the network synchronization.

\begin{corollary}\label{corollary}
Suppose $\alpha_i>V_i(t_0)$ and $0<d_i<1$ for $i=1,2,...,N$, then synchronization of network~\eqref{network-pinning} can be achieved, and this network with the proposed event-triggered impulses does not exhibit Zeno behavior.
\end{corollary}

\begin{remark}
This corollary is a direct result of Theorem~\ref{theorem} with $l=N$. Since the event-trigger~\eqref{event} is independently defined for each node, only part of the network nodes are controlled at each event time. This is the major advantage of the proposed control method over the one in~\cite{KZ-EB:2022} where all the nodes are required to be controlled by every impulse simultaneously. Furthermore, if $l<N$, another advantage over the event-triggered control method in~\cite{KZ-EB:2022} is that a pinning mechanism is incorporated with the impulsive control approach, and hence $N-l$ nodes are free of impulses over the entire time span of the network evolution. Condition~\eqref{condition2} requires the coupling strength $c$ is large enough, i.e., $c> \gamma/|\lambda_{max}(\bar{A})|$, so that synchronization of these impulse-free nodes can be achieved through the network connections with the pinning nodes.
\end{remark}

Next, we generalize the proposed control method to deal with synchronization problem of network~\eqref{network-pinning} with an adaptive coupling strength, that is, $c$ is replaced by $c(t)$ with the following adaptive law:
\begin{align}\label{adaptivelaw}
\dot{c}(t)&=\zeta \sum^N_{j=l+1} \left( x_j(t)-z(t) \right)^\top\Gamma\left( x_j(t)-z(t) \right) \cr
c(t_0)&=c_0
\end{align}
where $c_0\geq 0$ and $\zeta>0$ are arbitrary constants.
\begin{theorem}\label{theorem2}
Suppose Assumption~\ref{Assumption} holds, then network~\eqref{network-pinning} with the adaptive coupling strength and adaptive law~\eqref{adaptivelaw} achieves synchronization.
\end{theorem}
\begin{proof} According to the adaptive law~\eqref{adaptivelaw}, coupling strength $c(t)$ is strictly increasing. If $c(t)\rightarrow \infty$ as $t\rightarrow \infty$, then there exists a large enough time $T>t_0$ such that $c(t)>\frac{\gamma}{|\lambda_{max}(\bar{A})|}$ for all $t> T$, and then condition~\eqref{condition2} holds, and Theorem~\ref{theorem} implies the network synchronization can be achieved. Hence, it is sufficient to consider bounded coupling strength. In what follows, we suppose $c(t)<c_{max}$ for $t\geq t_0$, where $c_{max}>0$ is a constant. Consider the following Lyapunov-like function for these $N-l$ uncontrolled nodes
\begin{equation}
W(t)=\bar{e}^\top(t) \bar{e}(t) +\frac{\eta}{\zeta}(c(t)-\bar{c})^2 \left( 1+  \exp({-\check{\beta}(t-t_0)}) \right)
\end{equation}
where $\bar{e}=(e^\top_{l+1},e^\top_{l+2},...,e^\top_{N})^\top$, $\check{\beta}=\min_{1\leq i\leq N}\{\beta_i\}$, constants $\eta>0$ and $\bar{c}\geq c_{max}$ are to be determined. Similar to the discussion of~\eqref{derivative1}, differentiation of $W(t)$ along the trajectory of~\eqref{network-error} with adaptive coupling strength $c(t)$ gives
\begin{align}\label{dW}
\dot{W}(t) =& 2\bar{e}^\top \dot{\bar{e}} + 2\frac{\eta}{\zeta} (c(t)-\bar{c})\dot{c}(t)  \left( 1+\exp({-\check{\beta}(t-t_0)}) \right)\cr
     & -\frac{\check{\beta}\eta}{\zeta} (c(t)-\bar{c})^2    \exp({-\check{\beta}(t-t_0)} )  \cr
\leq & 2\bar{e}^\top \left( \left( \gamma I_{N-l} + c(t) \left(\bar{A}+2\eta I_{n-l}\right) - \eta\bar{c}I_{n-l} \right)\otimes\Gamma +\varepsilon I_{n(N-l)} \right) \bar{e}\cr
     & +\hat{\alpha} \exp({-\check{\beta}(t-t_0)})-\frac{\check{\beta}\eta}{\zeta} (c(t)-\bar{c})^2    \exp({-\check{\beta}(t-t_0)})
\end{align}
where $\varepsilon>0$ and
\[
\hat{\alpha}=\frac{1 }{2\varepsilon} (N-l)c_{max}^2 \gamma^2 l^2 \max_{{\substack{l+1\leq i\leq N \\ 1\leq j\leq l}}}\{a^2_{ij}\}  \max_{1\leq i\leq l}\{\alpha_i\}.
\]
Since $\bar{A}<0$, there exists a small enough $\eta>0$ such that $\bar{A}+2\eta I_{n-l}<0$. From the fact that $c(t)\leq c_{max}$, we can then make $\bar{c}$ sufficiently large so that both
\begin{equation}\label{dW1}
\left( \gamma I_{N-l} + c(t) \left(\bar{A}+2\eta I_{n-l}\right) - \eta\bar{c}I_{n-l} \right)\otimes\Gamma +\varepsilon I_{n(N-l)}<0
\end{equation}
and
\begin{equation}\label{dW2}
\frac{\check{\beta}\eta}{\zeta} (c(t)-\bar{c})^2 \geq \hat{\alpha} ~~\textrm{for all}~t\geq t_0.
\end{equation}
hold. We then can conclude from~\eqref{dW} with~\eqref{dW1} and~\eqref{dW2}  that there exists a small $\mu>0$ such that
\begin{equation}\label{derivativeLyapunov}
\dot{W}(t)\leq -\mu \bar{e}^\top(t) \bar{e}(t) \textrm{ for }t\geq t_0.
\end{equation}
Define $g(t):=\int^t_{t_0} \bar{e}^\top(s)\bar{e}(s)\mathrm{d}s$ which is a positive, increasing, and differentiable function. Integrating both sides of~\eqref{derivativeLyapunov} from $t_0$ to $t$ yields $0\leq W(t)\leq W(t_0)-\mu g(t)$, which implies $g$ is upper bounded, and hence, $\lim_{t\rightarrow\infty}g(t)=\int^{\infty}_{t_0}\bar{e}^\top(s)\bar{e}(s)\mathrm{d}s$ is finite. Furthermore, we can obtain from~\eqref{derivativeLyapunov} that $\bar{e}$ is bounded as $\dot{W}(t)\leq 0$. The boundedness of both $c(t)$ and $\bar{e}(t)$ concludes that  the second derivative $\ddot{g}(t)=2\bar{e}^\top(t)\dot{\bar{e}}(t)$ is bounded, and then $\dot{g}(t)$ is uniformly continuous. According to Barbalat's lemma~\cite[Lemma~4.2]{JJES-WL:1991}, we have $\dot{g}(t)\rightarrow 0$ as $t\rightarrow \infty$, that is, $\|e_j(t)\|\rightarrow 0$ as $t\rightarrow \infty$ for $j=l+1,l+2,...,N$.


Therefore, the above discussion with~\eqref{Vipin} concludes $\|e_i(t)\|\rightarrow 0$ as $t\rightarrow \infty$ for $i=1,2,...,N$.
\end{proof}

\begin{remark}\label{remarkTH2}
The adaptive law~\eqref{adaptivelaw} is inspired by but different from the one in~\cite{WY-GC-JL:2009}. In~\eqref{adaptivelaw}, only unpinned nodes are used, and the parameter $\zeta>0$ can be chosen arbitrarily. It should be mentioned that Theorem~\ref{theorem2} is not applicable if $c(t)$ is unbounded. Nevertheless, combination of Theorems~\ref{theorem} and~\ref{theorem2} in the following sense will provide an effective approach to synchronize the network while ensuring the boundedness of the coupling strength. First, the sufficient condition on the coupling strength, $c>\gamma/|\lambda_{max}(\bar{A})|$, can be obtained from Theorem~\ref{theorem}. Then, we monitor the coupling strength $c(t)$ with the adaptive law~\eqref{adaptivelaw}, and prescribe a small $\epsilon>0$. When $c(t)$ grows bigger than $\gamma/|\lambda_{max}(\bar{A})|$ and reaches the value of $\gamma/|\lambda_{max}(\bar{A})| + \epsilon$ at some time $t^*$, we keep $c(t)=\gamma/|\lambda_{max}(\bar{A})| + \epsilon$ for all $t\geq t^*$, and Theorem~\ref{theorem} implies the network synchronization. If $c(t)$ never attains the value of $\gamma/|\lambda_{max}(\bar{A})|$, then Theorem~\ref{theorem2} will guarantee the synchronization of network~\eqref{network-pinning} with adaptive law~\eqref{adaptivelaw}. Therefore, by enforcing $c(t)\leq \gamma/|\lambda_{max}(\bar{A})| + \epsilon$, smaller coupling strength may be preserved to achieve the network synchronization, and Zeno behavior can be excluded as discussed in the proof of Theorem~\ref{theorem} since $c(t)$ and all the error states are bounded. See the example in Section~\ref{Sec4} for a demonstration of the adaptive $c(t)$. Furthermore, Theorem~\ref{theorem2} is independent of the network topology and the pinning scheme. Therefore, the above combination approach guarantees the network synchronization by arbitrarily pinning a fraction of the network nodes.
\end{remark}

\section{An Example}\label{Sec4}
In this section, we consider network~\eqref{network} with $n=3$, $N=8$, $c=8$, $\Gamma=\textrm{diag}(1,2,1)$, and
\begin{align}\label{examplef}
f(t,z(t))=\begin{bmatrix}
35z_2(t)-35z_1(t)\\
-7z_1(t)-z_1(t)z_3(t)+28z_2(t)\\
z_1(t)z_2(t)-3z_3(t)
\end{bmatrix}.
\end{align}
From~\cite{WY-GC-JL:2009}, we know that Assumption~\ref{Assumption} is satisfied with $K=\gamma I_3$ and $\gamma=30.9342$. The network topology described by the coupling matrix $A$ is given in Fig.~\ref{topology}. {It can be observed that Node 1 has the smallest degree which is smaller than $\gamma/c$, then node 1 should be pinned but $c<\gamma/|\lambda_{max}(\bar{A})|$ with the corresponding $\bar{A}$. Hence, more nodes should be controlled by the event-triggered impulses. Table~\ref{table} listed some examples of nodes selections. It can be seen that synchronization of network~\eqref{network-pinning} can be achieved through pinning control nodes 1, 2, and 5 by the proposed event-triggered impulses.

To illustrate the effectiveness of our theoretical results, we select the following parameters such that the sufficient conditions in Theorem~\ref{theorem} are satisfied. Let $\beta_1=0.8$, $\beta_2=0.6$, $\beta_5=0.9$, and $z(0)=(0.1, -0.2, 0.1)^\top$. Initial states $x_{ij}(0)$ ($i=1,2,...,8$ and $j=1,2,3$) and constants $d_1,d_2,d_5$ are selected uniformly and at random from interval $(-1,1)$, then let $\alpha_i=1.01\|x_i(0)-z(0)\|^2$ for $i=1,2,5$. Simulation results are illustrated in~Fig.~\ref{simulation}. The first figure in~Fig.~\ref{simulation} shows the Lyapunov function which verifies the network synchronization since $V(t)\rightarrow 0$ as $t\rightarrow \infty$. The network synchronization is also illustrated in the second figure with the error states~$e_{11},e_{22},e_{53}$. To have a clear view of the impulse effects, the third figure shows the magnified portions of these error states over the time period~$[5.5,6]$. According to the analysis in Table~\ref{table}, nodes 1, 2, and 5 are selected to be impulsively pinned. However, it can be observed that only nodes 1 and 5 are controlled over the time interval $[5.5,6]$, and the event times are different for these two nodes. The reason that node 2 is free of impulses over this time interval is that $V_2(t)$ stays below the threshold $\alpha_2 \exp\left({-\beta_2(t-t_0)}\right)$ for $5.5\leq t\leq 6$.

In the next simulation, Fig.~\ref{couplingstrength} shows the adaptive coupling strength $c(t)$ with different values of $\zeta$ in the adaptive law~\eqref{adaptivelaw}. It can be seen that $c(t)$ is bounded for both values of $\zeta$, which implies the network synchronization. However, $c(t)$ attains a much larger value than the required one $\gamma/|\lambda_{max}(\bar{A})|=7.2276$ when $\zeta=2$. For $\zeta=0.2$, $c(t)$ approaches $6.5568$ which is smaller than $\gamma/|\lambda_{max}(\bar{A})|$. Intuitively, small $\zeta$ corresponds to relatively small $\dot{c}(t)$, and hence the coupling strength grows slower, and a smaller bound of $c(t)$ may be approached. To ensure the coupling strength maintains a small value, it is desirable to enforce the coupling strength to be bounded by using the combination approach introduced in Remark~\ref{remarkTH2}, and then the upper bound of the coupling strength will not be bigger than  $7.2276$ plus an arbitrarily small constant $\epsilon>0$.


\begin{figure}[!t]
\centering
\includegraphics[width=3in]{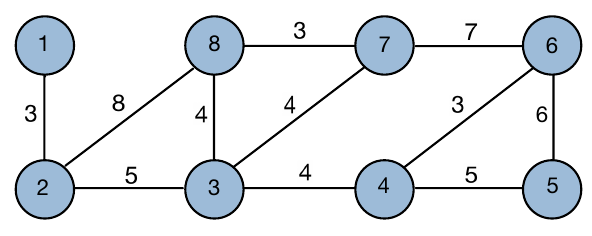}
\caption{Network topology.}
\label{topology}
\end{figure}

\begin{table}[!t] 
\caption{ }
\label{table}
\centering
\begin{tabular}{ ||c|c|c|c|c|| }
\hline 
& & & & \\[-.6em]
$l$ & Node Indexes  & $\lambda_{max}(\Bar{A})$ & $\frac{\gamma}{|\lambda_{max}(\Bar{A})|}$ & $c> \frac{\gamma}{|\lambda_{max}(\Bar{A})|}$ \\  [5pt]  \hline\hline
& & & & \\[-.7em]
1 & 1 & -0.3241 & 95.4465 & False \\ [2pt] \cline{1-5}  
& & & & \\[-.7em]
2 & 1,2 & -1.0968 & 28.2040 & False\\ [2pt] \cline{1-5}  
& & & & \\[-.7em]
2 & 1,5 & -1.7849 & 17.3311 & False\\ [2pt]  \cline{1-5}  
& & & & \\[-.7em]
3 & 1,4,5 & -2.4465 &  12.6443 & False \\ [2pt] \cline{1-5}  
& & & & \\[-.7em]
3 & 1,2,5 & -4.2800 &  7.2276 & \textbf{True}\\ [2pt] \cline{1-5} 
\hline 
\end{tabular} 
\end{table}

\begin{figure}[!t]
\centering
\includegraphics[width=3.5in]{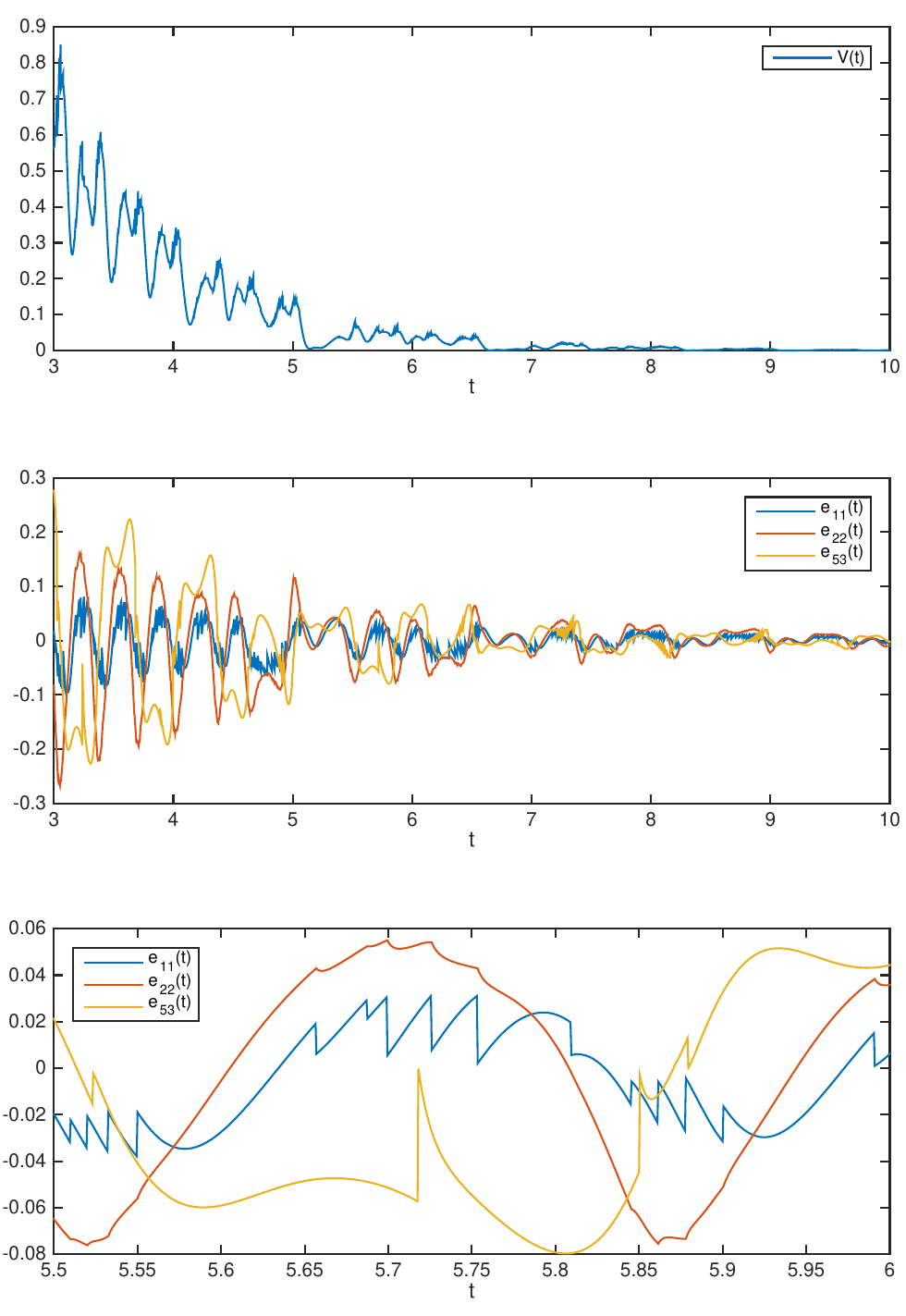}
\caption{Simulations of network~\eqref{network-pinning} with~\eqref{examplef}.}
\label{simulation}
\end{figure}

\begin{figure}[!t]
\centering
\includegraphics[width=3.5in]{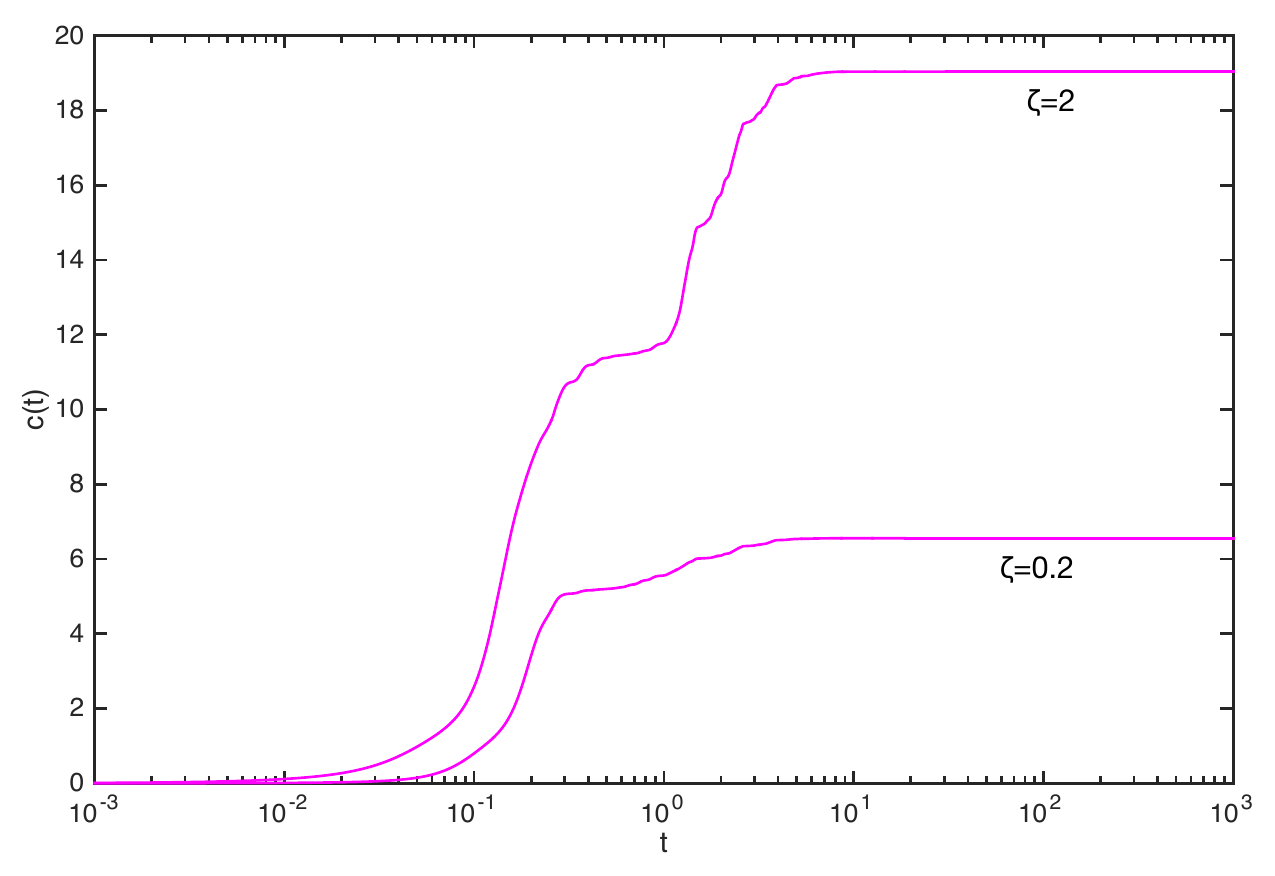}
\caption{Coupling strength $c(t)$ with adaptive law~\eqref{adaptivelaw}.}
\label{couplingstrength}
\end{figure}

\section{Conclusions}\label{Sec5}

Impulsive synchronization of complex dynamical networks has been investigated. An event-triggered pinning algorithm has been proposed for the impulsive controller, and sufficient conditions on the network topology have been established to realize the network synchronization. It has been shown that synchronization can also be achieved by arbitrary node selection with the proposed event-triggering algorithm and well-designed adaptive coupling strength. A simulation example has been discussed to verify the theoretical results. Future research could focus on the study of the proposed control method for the synchronization of networks with time delays in both the node dynamics and network connections.



\ifCLASSOPTIONcaptionsoff
  \newpage
\fi


%
%
%
%
%





\begin{thebibliography}{1}
\bibitem{EE:2012} 
E. Estrada, \emph{The structure of complex networks: theory and applications}. Oxford University Press, 2012.

\bibitem{CWW:2007}
C.W. Wu, \emph{Synchronization in complex networks of nonlinear dynamical systems}. World Scientific, 2007.


\bibitem{FD-FB:2014}
F. D\"orfler and F. Bullo, ``Synchronization in complex networks of phase oscillators: A survey", \emph{Automatica}, vol. 50, pp. 1539-1564, 2014.

\bibitem{YT-FQ-HG-JK:2014}
Y. Tang, F Qian, H. Gao, and J. Kurths, ``Synchronization in complex networks and its application - A survey of recent advances and challenges". \emph{Annual Reviews in Control}, vol. 38, no. 2, pp. 184-198, 2014.

\bibitem{TY:2001}
T. Yang, \emph{Impulsive control theory (Vol. 272)}. Springer Science \& Business Media, 2001.

\bibitem{BM-EYR:2003}
B. Miller and E.Y. Rubinovich, \emph{Impulsive control in continuous and discrete-continuous systems}. Springer Science \& Business Media, 2003.

\bibitem{XL-KZ:2019}
X. Liu and K. Zhang, \emph{Impulsive Systems on Hybrid Time Domains}. Springer, Cham, 2019.

\bibitem{KZ-EB:2022}
K. Zhang and E. Braverman, ``Event-triggered impulsive control for nonlinear systems with actuation delays". \emph{IEEE Transactions on Automatic Control}, vol. 68, no. 1, pp. 540-547, 2023.

\bibitem{XL-XY-JC:2020}
X. Li, X Yang, and J. Cao, ``Event-triggered impulsive control for nonlinear delay systems". \emph{Automatica}, vol. 117, 108981, 2020.

\bibitem{BJ-JL-XL-JQ:2021}
B. Jiang, J. Lu, X. Li, and J. Qiu, ``Event-triggered impulsive stabilization of systems with external disturbances". \emph{IEEE Transactions on Automatic Control}, vol. 67, no. 4, pp. 2116-2122, 2021.

\bibitem{XL-JC-XL-MA-UAA:2020}
X. Lv, J. Cao, X. Li, M. Abdel-Aty, and U.A. Al-Juboori, ``Synchronization analysis for complex dynamical networks with coupling delay via event-triggered delayed impulsive control". \emph{IEEE transactions on Cybernetics}, vol. 51, no. 11, pp. 5269-5278, 2020.

\bibitem{KU-FAR-XL-RR:2021}
K. Udhayakumar, F.A. Rihan, X. Li, and R. Rakkiyappan, ``Quasi-bipartite synchronisation of multiple inertial signed delayed neural networks under distributed event-triggered impulsive control strategy". \emph{IET Control Theory \& Applications}, vol. 15, no. 12, pp. 1615-1627, 2021.

\bibitem{SS-KU-DG-RR:2022}
S. Shanmugasundaram, K. Udhayakumar, D. Gunasekaran, and R. Rakkiyappan, ``Event-triggered impulsive control design for synchronization of inertial neural networks with time delays". \emph{Neurocomputing}, vol. 483, pp. 322-332, 2022.

\bibitem{ROG-MCC-GHS:1997}
R.O. Grigoriev, M.C. Cross, and H.G. Schuster, ``Pinning control of spatiotemporal chaos". \emph{Physical Review Letters}, vol. 79, no. 15, 2795, 1997.

\bibitem{WY-GC-JL:2009}
W. Yu, G. Chen, and J. L\"u, ``On pinning synchronization of complex dynamical networks". \emph{Automatica}, vol. 45, pp. 429-435, 2009.

\bibitem{YO-MJ-XY:2016}
Y. Orouskhani, M. Jalili, and X. Yu, ``Optimizing dynamical network structure for pinning control". \emph{Scientific Reports}, vol. 6, no. 1, pp. 1-13, 2016.

\bibitem{PD-FG-FLI:2018}
P. DeLellis, F. Garofalo, and F.L. Iudice, ``The partial pinning control strategy for large complex networks". \emph{Automatica}, vol. 89, pp. 111-116, 2018.

\bibitem{JL-JK-JC-NM-CH:2011}
J. Lu, J. Kurths, J. Cao, N. Mahdavi, and C. Huang, ``Synchronization control for nonlinear stochastic dynamical networks: pinning impulsive strategy". \emph{IEEE Transactions on Neural Networks and Learning Systems}, vol. 23, no. 2, pp. 285-292, 2011.

\bibitem{JZ-QW-LX:2011}
J. Zhou, Q. Wu, and L. Xiang, ``Pinning complex delayed dynamical networks by a single impulsive controller". \emph{IEEE Transactions on Circuits and Systems I: Regular Papers}, vol. 58, no. 12, pp. 2882-2893, 2011.

\bibitem{YS-XL:2022}
Y. Shen and X. Liu, ``Event-based master-slave synchronization of complex-valued neural networks via pinning impulsive control". \emph{Neural Networks}, vol. 145, pp. 374-385, 2022.

\bibitem{WL-SP-ZF-TC-ZG:2022}
W. Lin, S. Peng, Z. Fu, T. Chen, and Z. Gu, ``Consensus of fractional-order multi-agent systems via event-triggered pinning impulsive control". \emph{Neurocomputing}, vol. 494, pp. 409-417, 2022.

\bibitem{AA-FA-DL-GS-DVD-MDB-KHJ:2015}
A. Adaldo, F. Alderisio, D. Liuzza, G. Shi, D.V. Dimarogonas, M.D. Bernardo, and K.H. Johansson, ``Event-triggered pinning control of switching networks". \emph{IEEE Transactions on Control of Network Systems}, vol. 2, no. 2, pp. 204-213, 2015.

\bibitem{TW-KS-SS:2020}
T. Wakasa, K. Sawada, and S. Shin, ``Event-triggered switched pinning control for merging or splitting vehicle platoons". \emph{IFAC-PapersOnLine}, vol. 53, no. 2, pp. 15134-15139, 2020.

\bibitem{JJES-WL:1991}
J.-J. E. Slotine and W. Li, \emph{Applied Nonlinear Control}. Englewood Cliffs, New Jersey: Prentice-Hall, 1991. 


\end{thebibliography}
\end{document}